%% file: Cuntzsplicev3.tex
\documentclass[11pt]{article}
\usepackage{amsmath,amssymb,amsthm}

\newtheorem{theorem}{Theorem}
\newtheorem{lemma}[theorem]{Lemma}
\newtheorem{proposition}[theorem]{Proposition}
\newtheorem{corollary}[theorem]{Corollary}
\theoremstyle{definition}

\theoremstyle{plain}

\renewcommand{\labelenumi}{\textup{(\theenumi)}}

\title{A short note on Cuntz splice from a viewpoint of 
continuous orbit equivalence of topological Markov shifts}
\author{Kengo Matsumoto \\
Department of Mathematics \\
Joetsu University of Education \\
Joetsu, 943-8512, JAPAN
}
\date{}

\begin{document}
\maketitle

\def\det{{{\operatorname{det}}}}

\begin{abstract}
Let $A$ be an $N\times N$ irreducible matrix
 with entries in $\{0,1\}$.
We present an easy way to find an $(N+3)\times (N+3)$ irreducible matrix
$\bar{A}$ with entries in $\{0,1\}$ such that
their Cuntz--Krieger algebras 
${\mathcal{O}}_A$ 
and
${\mathcal{O}}_{\bar{A}}$
are isomorphic and  
$ \det(1 -A) = - \det(1-\bar{A}). $
As a consequence, we know that 
two Cuntz--Krieger algebras
${\mathcal{O}}_A$ 
and
${\mathcal{O}}_B$ 
are isomorphic if and only if 
the 
one-sided topological Markov shift
$(X_A, \sigma_A)$
is
continuously orbit equivalent to
either
$(X_B, \sigma_B)$
or
$(X_{\bar{B}}, \sigma_{\bar{B}}).$
\end{abstract}




\newcommand{\BF}{\operatorname{BF}}
\newcommand{\Hom}{\operatorname{Hom}}
\newcommand{\Homeo}{\operatorname{Homeo}}
\newcommand{\id}{\operatorname{id}}
\newcommand{\Ker}{\operatorname{Ker}}
\newcommand{\Ad}{\operatorname{Ad}}
\newcommand{\orb}{\operatorname{orb}}

\newcommand{\K}{\mathbb{K}}

\newcommand{\N}{\mathbb{N}}
\newcommand{\T}{\mathbb{T}}
\newcommand{\Z}{\mathbb{Z}}
\newcommand{\Zp}{{\mathbb{Z}}_+}
\def\BF{{{\operatorname{BF}}}}
\def\Ext{{{\operatorname{Ext}}}}
\def\Im{{{\operatorname{Im}}}}
\def\Max{{{\operatorname{Max}}}}
\def\Min{{{\operatorname{Min}}}}
\def\Aut{{{\operatorname{Aut}}}}
\def\Ad{{{\operatorname{Ad}}}}
\def\dim{{{\operatorname{dim}}}}
\def\det{{{\operatorname{det}}}}
\def\sgn{{{\operatorname{sgn}}}}
\def\supp{{{\operatorname{supp}}}}
\def\OB{{ {\mathcal{O}}_B}}
\def\OA{{ {\mathcal{O}}_A}}
\def\A{{ {\mathcal{A}}}}
\def\B{{ {\mathcal{B}}}}
\def\C{{ {\mathcal{C}}}}
\def\E{{ {\mathcal{E}}}}
\def\T{{ {\mathcal{T}}}}
\def\S{{ {\mathcal{S}}}}
\def\U{{ {\mathcal{U}}}}
\def\V{{ {\mathcal{V}}}}
\def\TG{{ {\frak T}_G }}
\def\AA{{ {\mathcal{A}}_A}}
\def\DB{{ {\mathcal{D}}_B}}
\def\DA{{ {\mathcal{D}}_A }}
\def\FA{{ {\mathcal{F}}_A }}
\def\E{{ {\mathcal{E}} }}

\def\OA{{{\mathcal{O}}_A}}
\def\OB{{{\mathcal{O}}_B}}
\def\FA{{{\mathcal{F}}_A}}
\def\FB{{{\mathcal{F}}_B}}
\def\DA{{{\mathcal{D}}_A}}
\def\DB{{{\mathcal{D}}_B}}
\def\HA{{{\frak H}_A}}
\def\HB{{{\frak H}_B}}
\def\Ext{{{\operatorname{Ext}}}}
\def\Max{{{\operatorname{Max}}}}
\def\Per{{{\operatorname{Per}}}}
\def\PerB{{{\operatorname{PerB}}}}
\def\Homeo{{{\operatorname{Homeo}}}}
\def\HSA{{H_{\sigma_A}(X_A)}}
\def\Out{{{\operatorname{Out}}}}
\def\Aut{{{\operatorname{Aut}}}}
\def\Ad{{{\operatorname{Ad}}}}
\def\Inn{{{\operatorname{Inn}}}}
\def\det{{{\operatorname{det}}}}
\def\exp{{{\operatorname{exp}}}}
\def\cobdy{{{\operatorname{cobdy}}}}
\def\Ker{{{\operatorname{Ker}}}}
\def\ind{{{\operatorname{ind}}}}
\def\id{{{\operatorname{id}}}}
\def\supp{{{\operatorname{supp}}}}
\def\COE{{{\operatorname{COE}}}}
\def\SCOE{{{\operatorname{SCOE}}}}



For an $N\times N$ irreducible matrix
$A$ with entries in $\{0,1\}$,
let us denote by $G(A)$ the abelian group
${\Z}^N /(1 - A^t){\Z}^N$
and 
by $u_A$ the position of the class 
$[(1,\dots,1)]$   
of the vector 
$(1,\dots,1)$
in the group $G(A)$.
Throughout this short note,
matrices are all assumed to be irreducible 
and not any permutation matrices.
J. Cuntz in \cite{C2} has shown that 
the pair
$(K_0(\OA), [1])$
of the 
$K_0$-group $K_0(\OA)$ 
of the Cuntz--Krieger algebra
$\OA$ 
and the class $[1]$ of the unit in 
$K_0(\OA)$ 
is isomorphic to 
$(G(A), u_A)$. 
In \cite{Ro}, 
 M. R{\o}rdam has shown that
$(G(A), u_A)$ is a complete invariant 
of the isomorphism class of $\OA$ (see \cite{EFW} for $N\le 3$).
For an $N\times N$ irreducible matrix
$A= [A(i,j)]_{i,j=1}^N$ with entries in $\{0,1\}$, 
the 
$(N+2)\times (N+2)$ irreducible matrix
$A_{-}$ defined by 
\begin{equation*}
A_{-}
=
\begin{bmatrix}
A(1,1)  & \dots & A(1,N-1)  & A(1,N)  & 0      &    0 \\  
\vdots  &       & \vdots    & \vdots  & \vdots & \vdots \\
A(N-1,1)& \dots & A(N-1,N-1)& A(N-1,N)& 0      &    0 \\  
A(N,1)  & \dots & A(N,N-1)  & A(N,N)  & 1      &    0 \\  
      0 & \dots & 0         & 1       & 1      &    1 \\  
      0 & \dots & 0         & 0       & 1      &    1    
\end{bmatrix} 
\end{equation*}
is called the Cuntz splice for $A$,
which has been first introduced in \cite{C86} by J. Cuntz,
related to classification problem for Cuntz--Krieger algebras.
In \cite{C86}, he had used the notation $A^-$ instead of the above $A_{-}$.
The crucial property of the Cuntz splice is that
$G(A_{-})$ is isomorphic to $G(A)$
and
$\det(1-A_{-}) = - \det(1-A)$.
The Cuntz splice 
$$
\begin{bmatrix}
1 & 1 & 0 & 0 \\
1 & 1 & 1 & 0 \\
0 & 1 & 1 & 1 \\
0 & 0 & 1 & 1   
\end{bmatrix}
$$
for
the matrix
$
\left[\begin{smallmatrix}
 1 & 1  \\
 1 & 1    
\end{smallmatrix}\right]
$
is denoted by $2_{-}$.
In the  proof of the above  R{\o}rdam's result \cite[Theorem 6.5]{Ro},
J. Cuntz's theorem \cite[Theorem 7.2]{Ro} is used 
which says that 
${\mathcal{O}}_2\cong{\mathcal{O}}_{2_{-}}$
implies 
${\mathcal{O}}_A\otimes\K\cong{\mathcal{O}}_{A_{-}}\otimes\K$
for all irreducible non-permutation matrices $A$.
Since R{\o}rdam has proved   
${\mathcal{O}}_2\cong{\mathcal{O}}_{2_{-}}$
(\cite[Lemma 6.4]{Ro}),
the result 
${\mathcal{O}}_A\otimes\K\cong{\mathcal{O}}_{A_{-}}\otimes\K$
holds for all irreducible non-permutation matrices $A$.
By using this result, R{\o}rdam has also obtained that
the group
$G(A)$ is a complete invariant of the stable isomorphism class 
of $\OA$.

Let us denote by $\BF(A)$
the abelian group
$G(A^t) = {\Z^N}/(1 -A){\Z^N}$,
which is called the Bowen--Franks group for 
$N\times N$ matrix $A$ (\cite{BF}).
Although $\BF(A)$ is isomorphic to $G(A)$ as a group,
there is no canonical isomorphism between them. 
Related to classification theory of symbolic dynamical systems,
J. Franks has shown that
the pair $( \BF(A), \sgn(\det(1-A)))$
is a complete invariant of the flow equivalence class of
the two-sided topological Markov shift
$(\bar{X}_A, \bar{\sigma}_A)$
by using Bown--Franks's result \cite{BF} 
for the group $\BF(A)$
and Parry--Sullivan's result
\cite{PS} for the determinant $\det(1-A)$.
Combining this with the  R{\o}rdam's result for the stable isomorphism classes of the Cuntz--Krieger algebras,   
$\OA$ is stably isomorphic to $\OB$
if and only if 
$(\bar{X}_A, \bar{\sigma}_A)$
is flow equivalent to
either
$(\bar{X}_B, \bar{\sigma}_B)$
or
$(\bar{X}_{B_{-}}, \bar{\sigma}_{B_{-}})$.

In \cite{MaPacific},
the author has introduced a notion of continuous orbit equivalence 
in  one-sided topological Markov shifts to classify Cuntz--Krieger algebras from a view point of topological dynamical system.
In \cite{MMKyoto}, 
H. Matui and the author have shown that  the triple
$(G(A), u_A, \sgn(\det(1- A)))$
is a complete invariant of the continuous orbit equivalence class
of the right one-sided topological Markov shift
$(X_A, \sigma_A)$.
This result is rephrased by using the above mentioned R{\o}rdam's result
for isomorphism classes of the Cuntz--Krieger algebras
such that 
the pair $(\OA, \sgn(\det(1- A)))$
is a complete invariant of the continuous orbit equivalence class
of the  one-sided topological Markov shift
$(X_A, \sigma_A)$.
The $C^*$-algebra $\mathcal{O}_{A_{-}}$
is not necessarily isomorphic to $\OA$,
whereas they are stably isomorphic,
because   
the position $u_{A_{-}}$ in $G(A_{-})$ generally 
is different from 
the position $u_{A}$ in $G(A)$. 
We note that the group $G(A)$ determines the absolute value
$| \det(1 -A)|$. 
If $G(A)$ is infinite, $\Ker(1-A)$ is not trivial so that $\det(1 - A)=0.$
If $G(A)$ is finite, it forms  a finite direct sum 
${\Z}/{m_1}{\Z} \oplus\cdots\oplus {\Z}/{m_r}{\Z}$ 
for some $m_1,\dots,m_r \in {\N}$ 
so that 
$|\det(1-A)| = m_1\cdots m_r$ (cf. \cite{C86}, \cite{CK}, \cite{Ro}).

By \cite[Lemma 3.7]{MMKyoto},
we know that there is a matrix
$A'$ with entries in $\{0,1\}$ such that 
the triple 
$(G(A), u_A, \sgn(\det(1-A)))$
is isomorphic to     
$(G(A'), u_{A'}, -\sgn(\det(1-A')))$,
which means that there exists an isomorphism
$\Phi:G(A)\rightarrow G(A')$
such that 
$\Phi(u_A) = u_{A'}$
and
$\sgn(\det(1- A)) = - \sgn(\det(1 - A'))$.
Following the given proof of 
\cite[Lemma 3.7]{MMKyoto},
the construction of the matrix $A'$ seems to be slightly complicated
and the matrix size of $A'$
  becomes much bigger than that of $A$.
It is not an easy task to present the matrix $A'$ for the given matrix $A$
in a concrete way.

In this short note,
we directly present 
an $(N+3)\times (N+3)$ matrix $\bar{A}$ 
with entries in $\{0,1\}$ such that
$(G(A), u_A, \sgn(\det(1-A)))$
is isomorphic to    
$(G(\bar{A}), u_{\bar{A}}, -\sgn(\det(1-\bar{A})))$.
The matrix $\bar{A}$ is constructed such that  
if $A$ is an irreducible non-permutation matrix,
so is $\bar{A}$.


We define
\begin{equation*}
A^{\circ}
=
\begin{bmatrix}
 A(1,1) & \dots & A(1,N-1)  & A(1,N)  & 0      \\  
 \vdots &       & \vdots    & \vdots  & \vdots \\  
A(N-1,1)& \dots & A(N-1,N-1)& A(N-1,N)& 0      \\  
  0     & \dots & 0         & 0       & 1      \\  
A(N,1)  & \dots & A(N,N-1)  & A(N,N)  & 0    
\end{bmatrix}
\end{equation*}
and
\begin{equation}
\bar{A}= (A^{\circ})_{-}
=
\begin{bmatrix}
A(1,1)  & \dots & A(1,N-1)  & A(1,N)  & 0      & 0      &    0 \\  
\vdots  &       & \vdots    & \vdots  & \vdots & \vdots & \vdots \\
A(N-1,1)& \dots & A(N-1,N-1)& A(N-1,N)& 0      & 0      &    0 \\  
      0 & \dots & 0         & 0       & 1      & 0      &    0 \\  
A(N,1)  & \dots & A(N,N-1)  & A(N,N)  & 0      & 1      &    0 \\  
      0 & \dots & 0         & 0       & 1      & 1      &    1 \\  
      0 & \dots & 0         & 0       & 0      & 1      &    1    
\end{bmatrix}. \label{eq:Aminus}
\end{equation}
The operation $A \rightarrow A^{\circ}$ 
is nothing but an expansion defined by Parry--Sullivan in \cite{PS},
and preserves their determinant:  $\det(1-A) = \det(1 - A^{\circ})$.
The following figure is a graphical expression
of the matrix $\bar{A}$ from $A$.

\begin{figure}[htbp]
\begin{center}
\input{winfig102}
\end{center}
\caption{}
\end{figure}

We provide two lemmas.
The first one is seen in \cite{BF}.
The second one is seen in \cite{C86} and \cite{Ro} in a different form.
\begin{lemma}[{\cite[Theorem 1.3]{BF}}]
The map 
\begin{equation*}
\eta_A: (x_1,\dots,x_{N-1}, x_N,x_{N+1}) \in {\Z}^{N+1}
\rightarrow
(x_1,\dots,x_{N-1}, x_N +x_{N+1}) \in {\Z}^N
\end{equation*}
induces an isomorphism $\bar{\eta}_A$
from
$G(A^\circ)$ to $G(A)$
such that 
$\bar{\eta}_A([(1,\dots,1,0)]) = u_A$.
\end{lemma}
\begin{lemma}[{cf. \cite[Proposition 2]{C86}, \cite[Proposition 7.1]{Ro}}]
The map 
\begin{equation*}
\xi_A: (x_1,\dots,x_N) \in {\Z}^N
\rightarrow
(x_1,\dots,x_N, 0,0) \in {\Z}^{N+2}
\end{equation*}
induces an isomorphism $\bar{\xi}_A$
from
$G(A)$ to $G(A_{-})$
such that 
$\bar{\xi}_A([(1,\dots,1,0)]) = u_{A_{-}}$.
\end{lemma}
\begin{proof}
For $ y =(y_1,\dots,y_N) \in {\Z}^N$,
put
$$
z = 
\begin{bmatrix}
z_1\\
\vdots\\
z_N
\end{bmatrix}
= (1 -A^t)
\begin{bmatrix}
y_1\\
\vdots\\
y_N
\end{bmatrix}.
$$
We then have
$$
\xi_A(z) = 
\begin{bmatrix}
z_1\\
\vdots\\
z_N \\
0\\
0
\end{bmatrix}
= (1 -A_{-}^t)
\begin{bmatrix}
y_1\\
\vdots\\
y_N \\
0 \\
-y_N
\end{bmatrix}.
$$
Hence we have
$\xi_A((1 - A^t){\Z}^N) \subset (1 - A_{-}^t){\Z}^{N+2}$
so that 
$\xi_A: {\Z}^N \rightarrow {\Z}^{N+2}$
induces a homomorphism 
from $G(A)$ to $G(A_{-})$
denoted by $\bar{\xi}_A$.
Suppose that
$[\xi(x_1,\dots,x_N)] =0 $ in $G(A_{-})$
so that
$$
\begin{bmatrix}
x_1\\
\vdots\\
x_N \\
0 \\
0
\end{bmatrix}
= (1 -A_{-}^t)
\begin{bmatrix}
z_1\\
\vdots\\
z_N \\
z_{N+1} \\
z_{N+2}
\end{bmatrix}
$$
for some 
$(z_1,\dots,z_{N+2}) \in {\Z}^{N+2}$.
It then follows that 
$z_{N+1}=0, z_{N+2} = -z_N$ so that 
$$
\begin{bmatrix}
x_1\\
\vdots\\
x_N 
\end{bmatrix}
= (1 -A^t)
\begin{bmatrix}
z_1\\
\vdots\\
z_N 
\end{bmatrix}.
$$
This implies 
$[(x_1,\dots,x_N)] =0 $ in $G(A)$
and hence $\bar{\xi}_A$ is injective.

For 
$(x_1,\dots, x_N,x_{N+1}, x_{N+2}) \in {\Z}^{N+2}$,
we have 
$$
\begin{bmatrix}
x_1\\
\vdots\\
x_N \\
x_{N+1} \\
x_{N+2} 
\end{bmatrix}
=
\begin{bmatrix}
x_1\\
\vdots\\
x_{N-1} \\
x_N - x_{N+2} \\
0 \\
0 
\end{bmatrix}
+
\begin{bmatrix}
0\\
\vdots\\
0 \\
x_{N+2} \\
x_{N+1} \\
x_{N+2} 
\end{bmatrix}
=
\begin{bmatrix}
x_1\\
\vdots\\
x_{N-1} \\
x_N - x_{N+2} \\
0 \\
0 
\end{bmatrix}
+
 (1 - A_{-}^t)
\begin{bmatrix}
0  \\
\vdots\\
0   \\
-z_{N+2} \\
-z_{N+1}
\end{bmatrix}.
$$
This implies that
$
[(x_1,\dots,x_N, x_{N+1}, x_{N+2})]
=
\bar{\xi}_A([(x_1,\dots,x_{N-1}, x_N - x_{N-2})])
$
in
$G(A_{-})$.
Therefore  $\bar{\xi}_A: G(A) \rightarrow G(A_{-})$
is surjective and hence an isomorphism.
In particular, we see that
$
[(1,\dots,1,1,1)]
=
\bar{\xi}_A([(1,\dots,1, 0)])
$
in
$G(A_{-})$.
\end{proof}
We have the following theorem by the preceding two lemmas.
\begin{theorem}
For an $N\times N$ matrix $A$ 
with entries in $\{0,1\}$,
let
$\bar{A}$
be
the
$(N+3)\times (N+3)$ matrix  with entries in $\{0,1\}$
defined in \eqref{eq:Aminus}.
Then there exists an isomorphism
$\Phi: G(A) \rightarrow G(\bar{A})$
such that $\Phi(u_A) = u_{\bar{A}}$
and
the matrices $A, \bar{A}$ satisfy
$\det(1-A) = - \det(1 - \bar{A})$.
If $A$ is an irreducible non-permutation matrix,
so is $\bar{A}$.
\end{theorem}
\begin{proof}
Define
$\Phi: G(A) \rightarrow G(\bar{A})$
by 
$\Phi = \bar{\xi}_{A^\circ}\circ\bar{\eta}_A^{-1}$
so that
 $\Phi(u_A) 
=\bar{\xi}_{A^\circ}([(1,\dots,1,0)]) =u_{\bar{A}}$.
Since
$\det(1 - \bar{A}) = - \det(1 - A^\circ) 
=-\det(1-A),$
we see the desired assertion.
\end{proof}

Let $P$ be an $N \times N$ permutation matrix coming from a permutation 
of the set $\{1,2,\dots,N\}$.
Since  there exists a natural isomorphism
$\Phi_P: G(A) \longrightarrow G(PAP^{-1})$ such that 
$\Phi_P(u_A) = u_{PAP^{-1}}$ and $\det( 1- A) = \det(1 -PAP^{-1})$,
the triplet
$(G(A), u_A,\det(1-A))$
does not depend on the choice of the vertex $v_N$ in the directed graph of the matrix $A$.

We have some corollaries.
\begin{corollary}
Let
$A$ be an irreducible non-permutation matrix with entries in $\{0,1\}$.
Then
$\OA$ is isomorphic to
${\mathcal{O}}_{\bar{A}}$
and
$\det(1-A) = - \det(1 - \bar{A})$.
\end{corollary}
Let $\bar{1}$ denote the matrix
$$
\begin{bmatrix}
0 & 1 & 0 & 0 \\
1 & 0 & 1 & 0 \\
0 & 1 & 1 & 1 \\
0 & 0 & 1 & 1   
\end{bmatrix}
$$
which is the matrix $\bar{A}$ for the $1\times 1$ matrix
$A = [1]$.
By the above theorem, we have
\begin{corollary}
$(K_0({\mathcal{O}}_{\bar{1}}),u_{\bar{1}})
=(\Z, 1)$. 
\end{corollary}
Hence the simple purely infinite $C^*$-algebra 
${\mathcal{O}}_{\bar{1}}$ has the same K-theory as 
the $C^*$-algebra 
${\mathcal{O}}_{1} = C(S^1)$ 
of the continuous functions on the unit circle $S^1$
with the positions of their units,
whereas 
$(K_0({\mathcal{O}}_{{1_{-}}}),u_{{1_{-}}})
=(\Z, 0)$
for the matrix
$
1_{-}
=
\left[\begin{smallmatrix}
 1 & 1 & 0 \\
 1 & 1 & 1 \\
 0 & 1 & 1   
\end{smallmatrix}\right]
$
by \cite{EFW}
(cf. \cite[p. 150]{C86}).

The following corollary has been shown in \cite{MMKyoto}. 
Its proof is now easy by using \cite{Ro}.
\begin{corollary}[{\cite[Lemma 3.7]{MMKyoto}}]
Let $F$ be  a finitely generated abelian group
and $u$ an element of $F$.
Let $s=0$ when $F$ is infinite
and $s= -1$ or $1$ when $F$ is finite.
Then there exists an irreducible non-permutation matrix
$A$ such that 
\begin{equation*}
(F,u,s) = (G(A), u_A, \sgn(\det(1 - A)).
\end{equation*}   
\end{corollary}
\begin{proof}
By \cite[Proposition 6.7 (i)]{Ro},
we know that 
there exists an irreducible non-permutation matrix
$A$ such that 
$(F,u) = (G(A), u_A).
$
If $s = \sgn(\det(1 - A))$, 
the matrix $A$ is the desired one,
otherwise $\bar{A}$ is the desired one.
\end{proof}
Let $A$ and $B$ be two irreducible non-permutation matrices with entries in 
$\{0,1\}$.
The one-sided topological Markov shifts
$(X_A,\sigma_A)$
and
$(X_B,\sigma_B)$
are said to be 
{\it flip}\/ continuously orbit  equivalent
if
$(X_A,\sigma_A)$
is continuously orbit equivalent to
either
$(X_B,\sigma_B)$
or
$(X_{\bar{B}},\sigma_{\bar{B}})$.
Similarly  
two-sided topological Markov shifts
$(\bar{X}_A, \bar{\sigma}_A)$
and
$(\bar{X}_B, \bar{\sigma}_B)$
are  said to be
{\it flip}\/  flow 
equivalent
if $(\bar{X}_A, \bar{\sigma}_A)$
 is flow equivalent to 
either
$(\bar{X}_B, \bar{\sigma}_B)$,
or 
$(\bar{X}_{\bar{B}}, \bar{\sigma}_{\bar{B}})$.
We thus have the following corollaries.
\begin{corollary}
Let
$A, B$ be  irreducible and not any permutation matrices with entries in $\{0,1\}$.
\begin{enumerate}
\renewcommand{\theenumi}{\roman{enumi}}
\renewcommand{\labelenumi}{\textup{(\theenumi)}}
\item
$\OA$ is isomorphic to
$\OB$
if and only if the one-sided topological Markov shifts
$(X_A,\sigma_A)$
and
$(X_B,\sigma_B)$
are  flip continuously orbit  equivalent.
\item 
$\OA$ is stably isomorphic to
$\OB$
if and only if
the two-sided topological Markov shifts
$(\bar{X}_A, \bar{\sigma}_A)$
and
$(\bar{X}_B, \bar{\sigma}_B)$
are  flip flow equivalent.
\end{enumerate}
\end{corollary}

Let us denote by $[\OA]$ 
the isomorphism class of the Cuntz--Krieger algebra $\OA$ as a $C^*$-algebra.
Since $(G(A), u_A)$ is isomorphic to $(G(\bar{A}),u_{\bar{A}})$, 
we have $[\OA] = [{\mathcal{O}}_{\bar{A}}]$. 
We regard the sign  $\sgn(\det(1 -A))$ of $\det(1-A)$
as the orientation of the class $[\OA]$.
Then we can say that the pair $([\OA], \sgn(\det(1 -A)))$
is a complete invariant of the continuous orbit equivalence class of 
the one-sided topological Markov shift
$(X_A,\sigma_A)$.

In the rest of this short note, we present another square matrix $\tilde{A}$ 
of size $N+3$ from a square matrix $A=[A(i,j)]_{i,j=1}^N$ of size $N$  such that  
$\OA$ is isomorphic to ${\mathcal{O}}_{\tilde{A}}$ and
$\det( 1-A) = - \det(1 -\tilde{A})$. 
Define $(N+3) \times (N+3)$ matrix $\tilde{A}$
by setting
\begin{equation}
\tilde{A}
=
\begin{bmatrix}
A(1,1)  & \dots & A(1,N-1)  & A(1,N)  & 0      & 0      &    0 \\  
\vdots  &       & \vdots    & \vdots  & \vdots & \vdots & \vdots \\
A(N-1,1)& \dots & A(N-1,N-1)& A(N-1,N)& 0      & 0      &    0 \\  
      0 & \dots & 0         & 0       & 1      & 0      &    0 \\  
A(N,1)  & \dots & A(N,N-1)  & A(N,N)  & 0      & 1      &    0 \\  
      0 & \dots & 0         & 0       & 1      & 0      &    1 \\  
      0 & \dots & 0         & 0       & 0      & 1      &    1    
\end{bmatrix}. \label{eq:Atilde}
\end{equation}
The difference between the previous matrix $\bar{A}$ in \eqref{eq:Aminus}
 and the above matrix 
$\tilde{A}$ is the only $((N+2), (N+2))$-component.
Its graphical expression
of the matrix $\tilde{A}$ from $A$
is the following figure.
\begin{figure}[htbp]
\begin{center}
\input{winfig103}
\end{center}
\caption{}
\end{figure}

By virtue of  \cite{EFW}, we know the following proposition.
\begin{proposition}
The Cuntz--Krieger algebras ${\mathcal{O}}_{\bar{A}}$
and
${\mathcal{O}}_{\tilde{A}}$
are isomorphic, and
$\det(1 -\bar{A}) = \det( 1- \tilde{A})$.
\end{proposition}
\begin{proof}
Let us denote by
$\bar{A}_{i}$ the $i$th row vector of the matrix $\bar{A}$ of size $N+3$.
We put $E_i$ the row vector of size $N+3$ such that
$E_i = (0,\dots,0,\overset{i}{1},0,\dots,0)$
where the $i$th component is one, and the other components are zero.
Then we have
$\bar{A}_{N+2} = E_{N+1} + \bar{A}_{N+3}$.
Since the $(N+2)$th row ${\tilde{A}}_{N+2}$ of $\tilde{A}$
is  
$\tilde{A}_{N+2} = E_{N+1} + E_{N+3}$,
and the other rows of $\tilde{A}$ are the same as those of $\bar{A}$,
the matrix $\tilde{A}$ is obtained from $\bar{A}$ by the primitive transfer 
$$
\bar{A} \quad 
\underset{E_{N+1} + \bar{A}_{N+3} \rightarrow \tilde{A}_{N+2}}{\Longrightarrow}
\quad \tilde{A}
$$
in the sense of \cite[Definition 3.5]{EFW}.
We obtain that 
  ${\mathcal{O}}_{\bar{A}}$
is isomorphic to
${\mathcal{O}}_{\tilde{A}}$
by \cite[Theorem 3.7]{EFW}, and
$\det(1 -\bar{A}) = \det( 1- \tilde{A})$
by \cite[Theorem 8.4]{EFW}.
\end{proof}

Before ending this short note, we  refer to differences among  
the three  matrices $A_{-}, \bar{A}, \tilde{A}$ 
from a view point of dynamical system.
As $(G(A_{-}), \det(1 - A_{-}))
=(G(\bar{A}), \det(1 - \bar{A}))=(G(\tilde{A}), \det(1 - \tilde{A}))$,
there is a possibility that their  two sided topological Markov shifts
$(\bar{X}_{A_{-}}, \bar{\sigma}_{A_{-}}),
(\bar{X}_{\bar{A}}, \bar{\sigma}_{\bar{A}}),
(\bar{X}_{\tilde{A}}, \bar{\sigma}_{\tilde{A}})$
are topologically conjugate.
We however know 
that 
they are not topologically conjugate to each other in general
by the following example.
Denote by $p_n(\bar{\sigma}_A)$ the cardinal number of the $n$-periodic points
$\{ x \in \bar{X}_A \mid \bar{\sigma}_A^n(x) =x \}$ of the topological Markov shift
$(\bar{X}_A, \bar{\sigma}_A)$.
The zeta function $\zeta_A(z)$
for $(\bar{X}_A, \bar{\sigma}_A)$
is defined by 
\begin{equation*}
\zeta_A(z) = \exp\left( \sum_{n=1}^\infty \frac{p_n(\bar{\sigma}_A)}{n}z^n\right)
\qquad
(cf. \cite{LM}).
\end{equation*}
It is well-known that the formula
$\zeta_A(z) = \frac{1}{\det(1-zA)}$ 
holds (\cite{BL}).
Let us denote by 
$2_{-}, \bar{2}, \tilde{2}$ the matrices
$A_{-}, \bar{A}, \tilde{A}$
for
$
\left[\begin{smallmatrix}
1 & 1  \\
1 & 1 
\end{smallmatrix}\right]
$
respectively.
It is direct to see that 
\begin{gather*}
\zeta_{2_{-}}(z)
= \frac{1}{1 - 4z + 3 z^2 + 2 z^3 - z^4}, \\
\zeta_{\bar{2}}(z)=\frac{1}{ 1 - 3z + 4 z^3 - z^4}, \qquad
\zeta_{\tilde{2}}(z)=\frac{1}{ 1 - 3z + z^2 + z^3 + z^4}.
\end{gather*}
The zeta function is invariant under topological conjugacy so that 
$(\bar{X}_{2_{-}}, \bar{\sigma}_{2_{-}}), 
(\bar{X}_{\bar{2}}, \bar{\sigma}_{\bar{2}}),
(\bar{X}_{\tilde{2}}, \bar{\sigma}_{\tilde{2}})$ 
are not topologically conjugate to each other.

\medskip

This paper is a revised version of the paper entitled 
``Continuous orbit equivalence of topological Markov shifts and Cuntz splice''
arXiv:1511.01193v2 [math.OA].

\medskip

{\it Acknowledgment.}
This work was supported by JSPS KAKENHI Grant Number 15K04896.

\end{document}

%% file: winfig102.tex
\unitlength 0.1in
\begin{picture}( 52.4000, 19.8000)(  2.3000,-27.4000)
%
\special{pn 8}%
\special{pa 2516 2010}%
\special{pa 2520 1980}%
\special{pa 2536 1952}%
\special{pa 2560 1930}%
\special{pa 2586 1910}%
\special{pa 2614 1894}%
\special{pa 2642 1880}%
\special{pa 2672 1868}%
\special{pa 2702 1860}%
\special{pa 2734 1852}%
\special{pa 2764 1842}%
\special{pa 2796 1836}%
\special{pa 2826 1830}%
\special{pa 2858 1826}%
\special{pa 2890 1820}%
\special{pa 2922 1816}%
\special{pa 2954 1814}%
\special{pa 2986 1812}%
\special{pa 3018 1810}%
\special{pa 3050 1810}%
\special{pa 3082 1810}%
\special{pa 3114 1810}%
\special{pa 3146 1810}%
\special{pa 3178 1810}%
\special{pa 3210 1814}%
\special{pa 3242 1816}%
\special{pa 3274 1820}%
\special{pa 3304 1824}%
\special{pa 3336 1830}%
\special{pa 3368 1836}%
\special{pa 3400 1840}%
\special{pa 3430 1848}%
\special{pa 3460 1860}%
\special{pa 3490 1870}%
\special{pa 3520 1882}%
\special{pa 3550 1894}%
\special{pa 3578 1910}%
\special{pa 3604 1928}%
\special{pa 3626 1950}%
\special{pa 3644 1978}%
\special{pa 3648 2006}%
\special{sp}%
%
\special{pn 8}%
\special{pa 2516 2000}%
\special{pa 2516 2010}%
\special{fp}%
\special{sh 1}%
\special{pa 2516 2010}%
\special{pa 2536 1944}%
\special{pa 2516 1958}%
\special{pa 2496 1944}%
\special{pa 2516 2010}%
\special{fp}%
%
\special{pn 8}%
\special{pa 3802 1968}%
\special{pa 3776 1948}%
\special{pa 3756 1924}%
\special{pa 3740 1896}%
\special{pa 3728 1868}%
\special{pa 3716 1836}%
\special{pa 3708 1806}%
\special{pa 3700 1776}%
\special{pa 3694 1744}%
\special{pa 3688 1712}%
\special{pa 3686 1680}%
\special{pa 3684 1648}%
\special{pa 3686 1616}%
\special{pa 3688 1584}%
\special{pa 3692 1552}%
\special{pa 3696 1520}%
\special{pa 3704 1490}%
\special{pa 3712 1458}%
\special{pa 3722 1428}%
\special{pa 3734 1398}%
\special{pa 3748 1370}%
\special{pa 3768 1346}%
\special{pa 3790 1322}%
\special{pa 3816 1302}%
\special{pa 3846 1296}%
\special{pa 3878 1300}%
\special{pa 3906 1316}%
\special{pa 3930 1338}%
\special{pa 3948 1364}%
\special{pa 3964 1392}%
\special{pa 3974 1422}%
\special{pa 3988 1450}%
\special{pa 3996 1482}%
\special{pa 4000 1514}%
\special{pa 4006 1544}%
\special{pa 4010 1576}%
\special{pa 4012 1608}%
\special{pa 4012 1640}%
\special{pa 4012 1672}%
\special{pa 4008 1704}%
\special{pa 4002 1736}%
\special{pa 3998 1768}%
\special{pa 3992 1798}%
\special{pa 3982 1830}%
\special{pa 3972 1860}%
\special{pa 3958 1888}%
\special{pa 3942 1916}%
\special{pa 3924 1942}%
\special{pa 3898 1962}%
\special{pa 3884 1970}%
\special{sp}%
%
\special{pn 8}%
\special{pa 3794 1962}%
\special{pa 3802 1968}%
\special{fp}%
\special{sh 1}%
\special{pa 3802 1968}%
\special{pa 3760 1912}%
\special{pa 3758 1936}%
\special{pa 3736 1944}%
\special{pa 3802 1968}%
\special{fp}%
\put(23.4100,-14.1100){\makebox(0,0)[lb]{$v_N$}}%
%
\special{pn 8}%
\special{pa 5248 1968}%
\special{pa 5222 1948}%
\special{pa 5204 1922}%
\special{pa 5186 1896}%
\special{pa 5172 1868}%
\special{pa 5162 1836}%
\special{pa 5152 1806}%
\special{pa 5146 1774}%
\special{pa 5140 1744}%
\special{pa 5136 1712}%
\special{pa 5132 1680}%
\special{pa 5132 1648}%
\special{pa 5134 1616}%
\special{pa 5136 1584}%
\special{pa 5138 1552}%
\special{pa 5144 1520}%
\special{pa 5150 1490}%
\special{pa 5160 1458}%
\special{pa 5170 1428}%
\special{pa 5180 1398}%
\special{pa 5194 1370}%
\special{pa 5214 1344}%
\special{pa 5238 1322}%
\special{pa 5264 1304}%
\special{pa 5294 1296}%
\special{pa 5326 1302}%
\special{pa 5354 1318}%
\special{pa 5376 1340}%
\special{pa 5396 1366}%
\special{pa 5412 1394}%
\special{pa 5424 1424}%
\special{pa 5434 1454}%
\special{pa 5442 1484}%
\special{pa 5450 1516}%
\special{pa 5454 1548}%
\special{pa 5456 1580}%
\special{pa 5458 1612}%
\special{pa 5458 1644}%
\special{pa 5456 1676}%
\special{pa 5454 1708}%
\special{pa 5450 1738}%
\special{pa 5446 1770}%
\special{pa 5438 1802}%
\special{pa 5428 1832}%
\special{pa 5418 1862}%
\special{pa 5404 1890}%
\special{pa 5388 1918}%
\special{pa 5368 1944}%
\special{pa 5342 1964}%
\special{pa 5332 1970}%
\special{sp}%
%
\special{pn 8}%
\special{pa 5242 1964}%
\special{pa 5248 1968}%
\special{fp}%
\special{sh 1}%
\special{pa 5248 1968}%
\special{pa 5208 1912}%
\special{pa 5206 1936}%
\special{pa 5182 1942}%
\special{pa 5248 1968}%
\special{fp}%
\put(22.7900,-21.7700){\makebox(0,0)[lb]{$v_{N+1}$}}%
%
\special{pn 8}%
\special{pa 2406 1524}%
\special{pa 2410 2024}%
\special{fp}%
\special{sh 1}%
\special{pa 2410 2024}%
\special{pa 2430 1958}%
\special{pa 2410 1972}%
\special{pa 2390 1958}%
\special{pa 2410 2024}%
\special{fp}%
%
\special{pn 8}%
\special{ar 2434 1368 180 116  0.0000000 6.2831853}%
\put(36.6800,-21.5400){\makebox(0,0)[lb]{$v_{N+2}$}}%
\put(51.3500,-21.4800){\makebox(0,0)[lb]{$v_{N+3}$}}%
%
\special{pn 8}%
\special{pa 2796 760}%
\special{pa 2474 1232}%
\special{fp}%
\special{sh 1}%
\special{pa 2474 1232}%
\special{pa 2528 1188}%
\special{pa 2504 1188}%
\special{pa 2496 1166}%
\special{pa 2474 1232}%
\special{fp}%
%
\special{pn 8}%
\special{pa 2050 760}%
\special{pa 2370 1232}%
\special{fp}%
\special{sh 1}%
\special{pa 2370 1232}%
\special{pa 2348 1166}%
\special{pa 2340 1188}%
\special{pa 2316 1188}%
\special{pa 2370 1232}%
\special{fp}%
%
\special{pn 8}%
\special{ar 2430 2136 184 118  0.0000000 6.2831853}%
%
\special{pn 8}%
\special{ar 3828 2100 182 116  0.0000000 6.2831853}%
%
\special{pn 8}%
\special{ar 5288 2098 184 116  0.0000000 6.2831853}%
%
\special{pn 8}%
\special{pa 3652 2208}%
\special{pa 3648 2240}%
\special{pa 3634 2270}%
\special{pa 3614 2294}%
\special{pa 3590 2314}%
\special{pa 3562 2332}%
\special{pa 3534 2346}%
\special{pa 3506 2362}%
\special{pa 3476 2372}%
\special{pa 3446 2380}%
\special{pa 3414 2390}%
\special{pa 3384 2398}%
\special{pa 3352 2404}%
\special{pa 3320 2410}%
\special{pa 3288 2414}%
\special{pa 3258 2418}%
\special{pa 3226 2422}%
\special{pa 3194 2424}%
\special{pa 3162 2426}%
\special{pa 3130 2426}%
\special{pa 3098 2426}%
\special{pa 3066 2426}%
\special{pa 3034 2424}%
\special{pa 3002 2424}%
\special{pa 2970 2422}%
\special{pa 2938 2418}%
\special{pa 2906 2412}%
\special{pa 2874 2408}%
\special{pa 2844 2400}%
\special{pa 2812 2394}%
\special{pa 2780 2388}%
\special{pa 2750 2378}%
\special{pa 2720 2368}%
\special{pa 2690 2356}%
\special{pa 2662 2342}%
\special{pa 2634 2326}%
\special{pa 2606 2308}%
\special{pa 2586 2286}%
\special{pa 2566 2260}%
\special{pa 2554 2230}%
\special{pa 2554 2214}%
\special{sp}%
%
\special{pn 8}%
\special{pa 3652 2218}%
\special{pa 3652 2208}%
\special{fp}%
\special{sh 1}%
\special{pa 3652 2208}%
\special{pa 3632 2276}%
\special{pa 3652 2262}%
\special{pa 3672 2276}%
\special{pa 3652 2208}%
\special{fp}%
\put(5.4600,-14.2100){\makebox(0,0)[lb]{$v_N$}}%
%
\special{pn 8}%
\special{ar 638 1380 180 114  0.0000000 6.2831853}%
%
\special{pn 8}%
\special{pa 1000 770}%
\special{pa 680 1242}%
\special{fp}%
\special{sh 1}%
\special{pa 680 1242}%
\special{pa 734 1198}%
\special{pa 710 1198}%
\special{pa 700 1176}%
\special{pa 680 1242}%
\special{fp}%
%
\special{pn 8}%
\special{pa 254 770}%
\special{pa 574 1242}%
\special{fp}%
\special{sh 1}%
\special{pa 574 1242}%
\special{pa 552 1176}%
\special{pa 544 1198}%
\special{pa 520 1198}%
\special{pa 574 1242}%
\special{fp}%
%
\special{pn 20}%
\special{pa 1216 1384}%
\special{pa 1792 1396}%
\special{fp}%
\special{sh 1}%
\special{pa 1792 1396}%
\special{pa 1726 1374}%
\special{pa 1740 1394}%
\special{pa 1726 1414}%
\special{pa 1792 1396}%
\special{fp}%
%
\special{pn 8}%
\special{pa 3974 2000}%
\special{pa 3978 1970}%
\special{pa 3996 1942}%
\special{pa 4018 1920}%
\special{pa 4044 1900}%
\special{pa 4072 1884}%
\special{pa 4100 1870}%
\special{pa 4130 1860}%
\special{pa 4160 1850}%
\special{pa 4192 1842}%
\special{pa 4222 1832}%
\special{pa 4254 1826}%
\special{pa 4286 1820}%
\special{pa 4316 1816}%
\special{pa 4348 1810}%
\special{pa 4380 1806}%
\special{pa 4412 1804}%
\special{pa 4444 1802}%
\special{pa 4476 1802}%
\special{pa 4508 1800}%
\special{pa 4540 1800}%
\special{pa 4572 1800}%
\special{pa 4604 1802}%
\special{pa 4636 1802}%
\special{pa 4668 1804}%
\special{pa 4700 1806}%
\special{pa 4732 1810}%
\special{pa 4764 1814}%
\special{pa 4794 1820}%
\special{pa 4826 1826}%
\special{pa 4858 1832}%
\special{pa 4888 1840}%
\special{pa 4920 1848}%
\special{pa 4950 1858}%
\special{pa 4980 1870}%
\special{pa 5008 1884}%
\special{pa 5038 1898}%
\special{pa 5062 1918}%
\special{pa 5086 1940}%
\special{pa 5102 1968}%
\special{pa 5108 1996}%
\special{sp}%
%
\special{pn 8}%
\special{pa 3974 1990}%
\special{pa 3974 2000}%
\special{fp}%
\special{sh 1}%
\special{pa 3974 2000}%
\special{pa 3994 1934}%
\special{pa 3974 1948}%
\special{pa 3954 1934}%
\special{pa 3974 2000}%
\special{fp}%
%
\special{pn 8}%
\special{pa 5096 2190}%
\special{pa 5092 2220}%
\special{pa 5078 2250}%
\special{pa 5056 2274}%
\special{pa 5032 2294}%
\special{pa 5006 2312}%
\special{pa 4978 2328}%
\special{pa 4948 2342}%
\special{pa 4918 2352}%
\special{pa 4888 2362}%
\special{pa 4858 2370}%
\special{pa 4826 2378}%
\special{pa 4794 2384}%
\special{pa 4764 2390}%
\special{pa 4732 2394}%
\special{pa 4700 2400}%
\special{pa 4668 2402}%
\special{pa 4636 2404}%
\special{pa 4604 2406}%
\special{pa 4572 2406}%
\special{pa 4540 2406}%
\special{pa 4508 2406}%
\special{pa 4476 2406}%
\special{pa 4444 2404}%
\special{pa 4412 2402}%
\special{pa 4380 2398}%
\special{pa 4348 2394}%
\special{pa 4318 2388}%
\special{pa 4286 2380}%
\special{pa 4256 2374}%
\special{pa 4224 2368}%
\special{pa 4192 2358}%
\special{pa 4162 2348}%
\special{pa 4132 2336}%
\special{pa 4104 2322}%
\special{pa 4078 2304}%
\special{pa 4050 2288}%
\special{pa 4028 2266}%
\special{pa 4008 2240}%
\special{pa 3998 2210}%
\special{pa 3998 2194}%
\special{sp}%
%
\special{pn 8}%
\special{pa 5096 2198}%
\special{pa 5096 2190}%
\special{fp}%
\special{sh 1}%
\special{pa 5096 2190}%
\special{pa 5076 2256}%
\special{pa 5096 2242}%
\special{pa 5116 2256}%
\special{pa 5096 2190}%
\special{fp}%
%
\special{pn 8}%
\special{pa 670 1524}%
\special{pa 992 1996}%
\special{fp}%
\special{sh 1}%
\special{pa 992 1996}%
\special{pa 970 1930}%
\special{pa 962 1952}%
\special{pa 938 1952}%
\special{pa 992 1996}%
\special{fp}%
%
\special{pn 8}%
\special{pa 552 1524}%
\special{pa 230 1996}%
\special{fp}%
\special{sh 1}%
\special{pa 230 1996}%
\special{pa 284 1952}%
\special{pa 260 1952}%
\special{pa 252 1930}%
\special{pa 230 1996}%
\special{fp}%
%
\special{pn 8}%
\special{pa 2466 2268}%
\special{pa 2786 2740}%
\special{fp}%
\special{sh 1}%
\special{pa 2786 2740}%
\special{pa 2766 2674}%
\special{pa 2756 2696}%
\special{pa 2732 2696}%
\special{pa 2786 2740}%
\special{fp}%
%
\special{pn 8}%
\special{pa 2368 2268}%
\special{pa 2046 2740}%
\special{fp}%
\special{sh 1}%
\special{pa 2046 2740}%
\special{pa 2100 2696}%
\special{pa 2076 2696}%
\special{pa 2068 2674}%
\special{pa 2046 2740}%
\special{fp}%
\end{picture}%

%% file: winfig103.tex
\unitlength 0.1in
\begin{picture}( 52.4000, 19.8000)(  2.3000,-27.4000)
%
\special{pn 8}%
\special{pa 2516 2010}%
\special{pa 2520 1980}%
\special{pa 2536 1952}%
\special{pa 2560 1930}%
\special{pa 2586 1910}%
\special{pa 2614 1894}%
\special{pa 2642 1880}%
\special{pa 2672 1868}%
\special{pa 2702 1860}%
\special{pa 2734 1852}%
\special{pa 2764 1842}%
\special{pa 2796 1836}%
\special{pa 2826 1830}%
\special{pa 2858 1826}%
\special{pa 2890 1820}%
\special{pa 2922 1816}%
\special{pa 2954 1814}%
\special{pa 2986 1812}%
\special{pa 3018 1810}%
\special{pa 3050 1810}%
\special{pa 3082 1810}%
\special{pa 3114 1810}%
\special{pa 3146 1810}%
\special{pa 3178 1810}%
\special{pa 3210 1814}%
\special{pa 3242 1816}%
\special{pa 3274 1820}%
\special{pa 3304 1824}%
\special{pa 3336 1830}%
\special{pa 3368 1836}%
\special{pa 3400 1840}%
\special{pa 3430 1848}%
\special{pa 3460 1860}%
\special{pa 3490 1870}%
\special{pa 3520 1882}%
\special{pa 3550 1894}%
\special{pa 3578 1910}%
\special{pa 3604 1928}%
\special{pa 3626 1950}%
\special{pa 3644 1978}%
\special{pa 3648 2006}%
\special{sp}%
%
\special{pn 8}%
\special{pa 2516 2000}%
\special{pa 2516 2010}%
\special{fp}%
\special{sh 1}%
\special{pa 2516 2010}%
\special{pa 2536 1944}%
\special{pa 2516 1958}%
\special{pa 2496 1944}%
\special{pa 2516 2010}%
\special{fp}%
\put(23.4100,-14.1100){\makebox(0,0)[lb]{$v_N$}}%
%
\special{pn 8}%
\special{pa 5248 1968}%
\special{pa 5222 1948}%
\special{pa 5204 1922}%
\special{pa 5186 1896}%
\special{pa 5172 1868}%
\special{pa 5162 1836}%
\special{pa 5152 1806}%
\special{pa 5146 1774}%
\special{pa 5140 1744}%
\special{pa 5136 1712}%
\special{pa 5132 1680}%
\special{pa 5132 1648}%
\special{pa 5134 1616}%
\special{pa 5136 1584}%
\special{pa 5138 1552}%
\special{pa 5144 1520}%
\special{pa 5150 1490}%
\special{pa 5160 1458}%
\special{pa 5170 1428}%
\special{pa 5180 1398}%
\special{pa 5194 1370}%
\special{pa 5214 1344}%
\special{pa 5238 1322}%
\special{pa 5264 1304}%
\special{pa 5294 1296}%
\special{pa 5326 1302}%
\special{pa 5354 1318}%
\special{pa 5376 1340}%
\special{pa 5396 1366}%
\special{pa 5412 1394}%
\special{pa 5424 1424}%
\special{pa 5434 1454}%
\special{pa 5442 1484}%
\special{pa 5450 1516}%
\special{pa 5454 1548}%
\special{pa 5456 1580}%
\special{pa 5458 1612}%
\special{pa 5458 1644}%
\special{pa 5456 1676}%
\special{pa 5454 1708}%
\special{pa 5450 1738}%
\special{pa 5446 1770}%
\special{pa 5438 1802}%
\special{pa 5428 1832}%
\special{pa 5418 1862}%
\special{pa 5404 1890}%
\special{pa 5388 1918}%
\special{pa 5368 1944}%
\special{pa 5342 1964}%
\special{pa 5332 1970}%
\special{sp}%
%
\special{pn 8}%
\special{pa 5242 1964}%
\special{pa 5248 1968}%
\special{fp}%
\special{sh 1}%
\special{pa 5248 1968}%
\special{pa 5208 1912}%
\special{pa 5206 1936}%
\special{pa 5182 1942}%
\special{pa 5248 1968}%
\special{fp}%
\put(22.7900,-21.7700){\makebox(0,0)[lb]{$v_{N+1}$}}%
%
\special{pn 8}%
\special{pa 2406 1524}%
\special{pa 2410 2024}%
\special{fp}%
\special{sh 1}%
\special{pa 2410 2024}%
\special{pa 2430 1958}%
\special{pa 2410 1972}%
\special{pa 2390 1958}%
\special{pa 2410 2024}%
\special{fp}%
%
\special{pn 8}%
\special{ar 2434 1368 180 116  0.0000000 6.2831853}%
\put(36.6800,-21.5400){\makebox(0,0)[lb]{$v_{N+2}$}}%
\put(51.3500,-21.4800){\makebox(0,0)[lb]{$v_{N+3}$}}%
%
\special{pn 8}%
\special{pa 2796 760}%
\special{pa 2474 1232}%
\special{fp}%
\special{sh 1}%
\special{pa 2474 1232}%
\special{pa 2528 1188}%
\special{pa 2504 1188}%
\special{pa 2496 1166}%
\special{pa 2474 1232}%
\special{fp}%
%
\special{pn 8}%
\special{pa 2050 760}%
\special{pa 2370 1232}%
\special{fp}%
\special{sh 1}%
\special{pa 2370 1232}%
\special{pa 2348 1166}%
\special{pa 2340 1188}%
\special{pa 2316 1188}%
\special{pa 2370 1232}%
\special{fp}%
%
\special{pn 8}%
\special{ar 2430 2136 184 118  0.0000000 6.2831853}%
%
\special{pn 8}%
\special{ar 3828 2100 182 116  0.0000000 6.2831853}%
%
\special{pn 8}%
\special{ar 5288 2098 184 116  0.0000000 6.2831853}%
%
\special{pn 8}%
\special{pa 3652 2208}%
\special{pa 3648 2240}%
\special{pa 3634 2270}%
\special{pa 3614 2294}%
\special{pa 3590 2314}%
\special{pa 3562 2332}%
\special{pa 3534 2346}%
\special{pa 3506 2362}%
\special{pa 3476 2372}%
\special{pa 3446 2380}%
\special{pa 3414 2390}%
\special{pa 3384 2398}%
\special{pa 3352 2404}%
\special{pa 3320 2410}%
\special{pa 3288 2414}%
\special{pa 3258 2418}%
\special{pa 3226 2422}%
\special{pa 3194 2424}%
\special{pa 3162 2426}%
\special{pa 3130 2426}%
\special{pa 3098 2426}%
\special{pa 3066 2426}%
\special{pa 3034 2424}%
\special{pa 3002 2424}%
\special{pa 2970 2422}%
\special{pa 2938 2418}%
\special{pa 2906 2412}%
\special{pa 2874 2408}%
\special{pa 2844 2400}%
\special{pa 2812 2394}%
\special{pa 2780 2388}%
\special{pa 2750 2378}%
\special{pa 2720 2368}%
\special{pa 2690 2356}%
\special{pa 2662 2342}%
\special{pa 2634 2326}%
\special{pa 2606 2308}%
\special{pa 2586 2286}%
\special{pa 2566 2260}%
\special{pa 2554 2230}%
\special{pa 2554 2214}%
\special{sp}%
%
\special{pn 8}%
\special{pa 3652 2218}%
\special{pa 3652 2208}%
\special{fp}%
\special{sh 1}%
\special{pa 3652 2208}%
\special{pa 3632 2276}%
\special{pa 3652 2262}%
\special{pa 3672 2276}%
\special{pa 3652 2208}%
\special{fp}%
\put(5.4600,-14.2100){\makebox(0,0)[lb]{$v_N$}}%
%
\special{pn 8}%
\special{ar 638 1380 180 114  0.0000000 6.2831853}%
%
\special{pn 8}%
\special{pa 1000 770}%
\special{pa 680 1242}%
\special{fp}%
\special{sh 1}%
\special{pa 680 1242}%
\special{pa 734 1198}%
\special{pa 710 1198}%
\special{pa 700 1176}%
\special{pa 680 1242}%
\special{fp}%
%
\special{pn 8}%
\special{pa 254 770}%
\special{pa 574 1242}%
\special{fp}%
\special{sh 1}%
\special{pa 574 1242}%
\special{pa 552 1176}%
\special{pa 544 1198}%
\special{pa 520 1198}%
\special{pa 574 1242}%
\special{fp}%
%
\special{pn 20}%
\special{pa 1216 1384}%
\special{pa 1792 1396}%
\special{fp}%
\special{sh 1}%
\special{pa 1792 1396}%
\special{pa 1726 1374}%
\special{pa 1740 1394}%
\special{pa 1726 1414}%
\special{pa 1792 1396}%
\special{fp}%
%
\special{pn 8}%
\special{pa 3974 2000}%
\special{pa 3978 1970}%
\special{pa 3996 1942}%
\special{pa 4018 1920}%
\special{pa 4044 1900}%
\special{pa 4072 1884}%
\special{pa 4100 1870}%
\special{pa 4130 1860}%
\special{pa 4160 1850}%
\special{pa 4192 1842}%
\special{pa 4222 1832}%
\special{pa 4254 1826}%
\special{pa 4286 1820}%
\special{pa 4316 1816}%
\special{pa 4348 1810}%
\special{pa 4380 1806}%
\special{pa 4412 1804}%
\special{pa 4444 1802}%
\special{pa 4476 1802}%
\special{pa 4508 1800}%
\special{pa 4540 1800}%
\special{pa 4572 1800}%
\special{pa 4604 1802}%
\special{pa 4636 1802}%
\special{pa 4668 1804}%
\special{pa 4700 1806}%
\special{pa 4732 1810}%
\special{pa 4764 1814}%
\special{pa 4794 1820}%
\special{pa 4826 1826}%
\special{pa 4858 1832}%
\special{pa 4888 1840}%
\special{pa 4920 1848}%
\special{pa 4950 1858}%
\special{pa 4980 1870}%
\special{pa 5008 1884}%
\special{pa 5038 1898}%
\special{pa 5062 1918}%
\special{pa 5086 1940}%
\special{pa 5102 1968}%
\special{pa 5108 1996}%
\special{sp}%
%
\special{pn 8}%
\special{pa 3974 1990}%
\special{pa 3974 2000}%
\special{fp}%
\special{sh 1}%
\special{pa 3974 2000}%
\special{pa 3994 1934}%
\special{pa 3974 1948}%
\special{pa 3954 1934}%
\special{pa 3974 2000}%
\special{fp}%
%
\special{pn 8}%
\special{pa 5096 2190}%
\special{pa 5092 2220}%
\special{pa 5078 2250}%
\special{pa 5056 2274}%
\special{pa 5032 2294}%
\special{pa 5006 2312}%
\special{pa 4978 2328}%
\special{pa 4948 2342}%
\special{pa 4918 2352}%
\special{pa 4888 2362}%
\special{pa 4858 2370}%
\special{pa 4826 2378}%
\special{pa 4794 2384}%
\special{pa 4764 2390}%
\special{pa 4732 2394}%
\special{pa 4700 2400}%
\special{pa 4668 2402}%
\special{pa 4636 2404}%
\special{pa 4604 2406}%
\special{pa 4572 2406}%
\special{pa 4540 2406}%
\special{pa 4508 2406}%
\special{pa 4476 2406}%
\special{pa 4444 2404}%
\special{pa 4412 2402}%
\special{pa 4380 2398}%
\special{pa 4348 2394}%
\special{pa 4318 2388}%
\special{pa 4286 2380}%
\special{pa 4256 2374}%
\special{pa 4224 2368}%
\special{pa 4192 2358}%
\special{pa 4162 2348}%
\special{pa 4132 2336}%
\special{pa 4104 2322}%
\special{pa 4078 2304}%
\special{pa 4050 2288}%
\special{pa 4028 2266}%
\special{pa 4008 2240}%
\special{pa 3998 2210}%
\special{pa 3998 2194}%
\special{sp}%
%
\special{pn 8}%
\special{pa 5096 2198}%
\special{pa 5096 2190}%
\special{fp}%
\special{sh 1}%
\special{pa 5096 2190}%
\special{pa 5076 2256}%
\special{pa 5096 2242}%
\special{pa 5116 2256}%
\special{pa 5096 2190}%
\special{fp}%
%
\special{pn 8}%
\special{pa 670 1524}%
\special{pa 992 1996}%
\special{fp}%
\special{sh 1}%
\special{pa 992 1996}%
\special{pa 970 1930}%
\special{pa 962 1952}%
\special{pa 938 1952}%
\special{pa 992 1996}%
\special{fp}%
%
\special{pn 8}%
\special{pa 552 1524}%
\special{pa 230 1996}%
\special{fp}%
\special{sh 1}%
\special{pa 230 1996}%
\special{pa 284 1952}%
\special{pa 260 1952}%
\special{pa 252 1930}%
\special{pa 230 1996}%
\special{fp}%
%
\special{pn 8}%
\special{pa 2466 2268}%
\special{pa 2786 2740}%
\special{fp}%
\special{sh 1}%
\special{pa 2786 2740}%
\special{pa 2766 2674}%
\special{pa 2756 2696}%
\special{pa 2732 2696}%
\special{pa 2786 2740}%
\special{fp}%
%
\special{pn 8}%
\special{pa 2368 2268}%
\special{pa 2046 2740}%
\special{fp}%
\special{sh 1}%
\special{pa 2046 2740}%
\special{pa 2100 2696}%
\special{pa 2076 2696}%
\special{pa 2068 2674}%
\special{pa 2046 2740}%
\special{fp}%
\end{picture}%